\documentclass[reqno]{amsart}
\usepackage[top=1.5in, bottom=1.5in, left=1in, right=1in]{geometry}
\usepackage{amssymb,bbm}
\usepackage{mathtools}
\usepackage{graphicx}
\usepackage{amsthm, calrsfs}
\usepackage[mathscr]{eucal}
\usepackage{rotating}
\usepackage{multirow,float,setspace}
\usepackage{color}
\numberwithin{equation}{section}
\newtheorem{theorem}{Theorem}

\newtheorem{corollary}{Corollary}

\theoremstyle{definition}
\newtheorem{definition}[theorem]{Definition}

\newtheorem{examples}[theorem]{Example}
\theoremstyle{remark}
\usepackage[colorlinks=true]{hyperref}
\newtheorem{remark}{Remark}
\usepackage{enumerate}

\newcommand{\RN}[1]{\uppercase\expandafter{\romannumeral#1}}

\begin{document}
	
\title[Growth Rates of Superlinear ODEs]{Growth Rates of Solutions of Superlinear Ordinary Differential Equations}
\author{John A. D. Appleby}
\address{School of Mathematical
	Sciences, Dublin City University, Glasnevin, Dublin 9, Ireland}
\email{john.appleby@dcu.ie} \urladdr{webpages.dcu.ie/\textasciitilde
	applebyj}

\author{Denis D. Patterson}
\address{School of Mathematical
	Sciences, Dublin City University, Glasnevin, Dublin 9, Ireland}
\email{denis.patterson2@mail.dcu.ie}
\urladdr{sites.google.com/a/mail.dcu.ie/denis-patterson}

\thanks{Denis Patterson is supported by the Irish Research Council grant GOIPG/2013/402.} 
\date{\today}
\begin{abstract}
In this letter we obtain sharp estimates on the growth rate of solutions to a nonlinear ODE with a nonautonomous forcing term. The equation is superlinear in the state variable and hence solutions exhibit rapid growth and finite--time blow--up. The importance of ODEs of the type considered here stems from the key role they play in understanding the asymptotic behaviour of more complex systems involving delay and randomness. 
\end{abstract}

\maketitle

\section{Introduction}
We study the asymptotic behaviour of rapidly growing solutions to the nonlinear ordinary differential equation
\begin{equation}\label{eq.ODE}
x'(t) = f(x(t)) + h(t), \quad t \geq 0; \quad x(0) = \psi>0.
\end{equation}
Rapid growth, and possibly even finite--time blow--up, of solutions is ensured by assuming
\begin{equation}\label{eq.f}
f \in C((0,\infty);(0,\infty)), \,\, f \mbox{ is increasing}, \,\, x\mapsto f_1(x):=f(x)/x\mbox{ is ultimately increasing,}\,\, \lim_{x\to\infty}f_1(x)= \infty .\tag{f}
\end{equation}
Note that \eqref{eq.f} \emph{precludes $f$ being subadditive} (cf. \cite{MR1487077}). Assuming $f$ is locally Lipschitz continuous is sufficient to ensure a unique solution to \eqref{eq.ODE} and, in order to simplify matters, we do so henceforth. We also assume
\begin{equation}\label{eq.H}
h \in C((0,\infty);\mathbb{R}), \quad H(t) = \int_{0}^t h(s)\,ds \geq 0 \mbox{ for each }t\geq 0. \tag{H}
\end{equation}
While understanding the asymptotics of \eqref{eq.ODE} is undoubtedly interesting in its own right, our primary interest in \eqref{eq.ODE} stems from the key role it plays in more complex systems exhibiting rapid growth. The asymptotic behaviour of blow--up solutions of nonlinear Volterra equations, such as
\begin{equation}\label{eq.volterra}
x'(t) = \int_{0}^t w(t-s)f(x(s))\,ds + h(t), \quad t \geq 0; \quad x(0) = \psi>0,
\end{equation}
have attracted considerable attention (see \cite{brunner2012blow,malolepszy2008blow,roberts2007recent} and the references therein). Of particular interest is the behaviour of solutions to \eqref{eq.volterra} in the key limit, if explosion occurs, or for large times, if solutions are global; the results of this letter for the simpler equation \eqref{eq.ODE} are an important first step in such an analysis (see e.g.~\cite{app_patt_ejde} for sublinear equations). Similarly, the nonlinear stochastic differential equation
\begin{equation}\label{eq.SDE}
X(t) = \psi+\int_0^t f(X(s))\,ds + \int_0^t \sigma(s)\,dB(s), \quad t \geq 0,
\end{equation}
can be studied using the results of this note (see Corollary \ref{corollary.SDE} and \cite{appleby2014necessary} for analysis in the sublinear case). Finally, we remark in the case that $a$ is a positive and continuous function, the non--autonomous ODE $z'(t)=a(t)f(z(t))+h(t)$ can be analysed by similar methods, since $\tilde{x}(t)=z(A^{-1}(t))$ obeys \eqref{eq.ODE}, where 
$A(t)=\int_0^t a(s)\,ds \to \infty$ as $t\to\infty$, and $\tilde{H}(t)=H(A^{-1}(t))$. Similar time--rescaling can be applied to 
non--autonomous analogues of \eqref{eq.SDE}.

Equation \eqref{eq.ODE} can be thought of as a perturbed version of the autonomous ODE
\begin{equation}\label{eq.ODE_autonomous}
y'(t) = f(y(t)), \quad t \geq 0;\quad y(0) = \psi > 0,
\end{equation}
whose solution is given by $y(t;\psi) = F^{-1}(\psi + t)$, where
$
F(x) = \int_1^x du/f(u) \mbox{ for } x \geq 1.
$
The function $F$ plays a central role in understanding the growth rate of solutions to \eqref{eq.ODE} since solutions to \eqref{eq.ODE_autonomous} obey 
\begin{equation}\label{eq.implicit_asymptotics}
\lim_{t\to\infty}\frac{F(y(t;\psi))}{t} = 1, \quad \mbox{for each }\psi>0,
\end{equation}
giving an implicit and $\psi$--independent estimate on the rate of growth. This should, and does, yield a more robust characterisation of the growth rate, since \eqref{eq.f} implies $\lim_{t\to\infty}y(t;\psi_1)/y(t;\psi_2)= 0$ for $\psi_1<\psi_2$. We prove necessary and sufficient conditions under which solutions to \eqref{eq.ODE} retain the implicit growth property \eqref{eq.implicit_asymptotics}. However, if  $h$ is sufficiently large, in an appropriate sense, we expect the solution to \eqref{eq.ODE} to grow at a rate determined by $h$; we show that this is the case by providing sharp conditions under which $\lim_{t\to\infty}x(t)/H(t) = 1$. 

This note is closely related to the vast literature on growth bounds of solutions of nonlinear differential and integral equations and inequalities (see e.g. \cite{MR1487077,agarwal2005generalization,bihari1956generalization,pinto1990integral}). However, it seems that applying the analysis of germane works in this area (e.g., \cite{MR1177923,MR1081393,MR0372135}) leads to weaker asymptotic results than we present here. In contrast to these works, our approach is a mixture of constructive comparison arguments (cf. e.g. \cite{appleby2014necessary}) and asymptotic integration methods (cf. e.g. \cite{app_patt_ejde,H}). Of course, since such works contend with more general problems under weaker assumptions, and establish global growth bounds, we should expect here to obtain sharper results under additional restrictions. We note that the monotonicity of $f_1$ implies $f$ obeys the reverse inequality to members of the class of functions $\mathcal{F}$, whose utility has been extensively exploited in the past (see \cite[Section 2.5]{MR1487077}).
\section{Main Results}
As is well--known, solutions to \eqref{eq.ODE} will be well--defined on $\mathbb{R}^+$ if and only if $\lim_{x\to\infty}F(x) = + \infty$. In the case when $\lim_{x\to\infty}F(x)<\infty$, the asymptotics of the solution to \eqref{eq.ODE} are given by the following result.
\begin{theorem} 
	Suppose $f \in C((0,\infty);(0,\infty))$ is increasing, $\lim_{x\to\infty}F(x)<\infty$ and \eqref{eq.H} holds. Then there is a $T \in (0,\infty)$ such that the solution to \eqref{eq.ODE} obeys $\lim_{t\to T^-}x(t) = \infty$, and  
	$\lim_{t\to T^-} (T-t)^{-1}\int_{x(t)}^\infty du/f(u)= 1$.
\end{theorem}
From this point on suppose $\lim_{x\to\infty}F(x) = \infty$, so solutions to \eqref{eq.ODE} are defined on $\mathbb{R}^+$. In  statements of subsequent results, $x$ is the unique continuous solution to \eqref{eq.ODE}, and this is henceforth omitted.  
\begin{theorem}\label{thm.lim_preserved}
	Suppose \eqref{eq.f} and \eqref{eq.H} hold. Then the following are equivalent:
	\[
	(i.)\quad \limsup_{t\to\infty}\frac{F(H(t))}{t} \in [0,1], \qquad (ii.)\quad \lim_{t\to\infty}\frac{F(x(t))}{t} = 1.
	\]
\end{theorem}
\begin{proof}[Proof of Theorem \ref{thm.lim_preserved}]
Integrate \eqref{eq.ODE} to obtain 
	$
	x(t) = x(0) + H(t) + \int_0^t f(x(s))\,ds, \, t \geq 0.
	$
	Hence, by \eqref{eq.H}, $x(t) \geq x(0) + \int_0^t f(x(s))\,ds$ for each $t \geq 0$. Define the lower comparison solution
	\[
	x_-(t) = x(0)/2 + \int_0^t f(x_-(s))\,ds \mbox{ for } t \geq 0.
\]
	By construction, $x_-(t) < x(t)$ for each $t \geq 0$ and furthermore, $x_-'(t) = f(x_-(t))$ for each $t>0$. Therefore, by asymptotic integration, $\lim_{t\to\infty}F(x_-(t))/t = 1$ and hence $\liminf_{t\to\infty}F(x(t))/t \geq 1$. Now suppose $\limsup_{t\to\infty}F(H(t))/t = K \in (0,1]$, postponing temporarily the case $K=0$. Thus, for each $\epsilon>0$, there exists $T(\epsilon)>0$ such that $H(t)< F^{-1}(K(1+\epsilon)t)$ for each $t \geq T(\epsilon)$. By integrating \eqref{eq.ODE}, derive the upper bound
	\[
	x(t) < x(0) + F^{-1}(K(1+\epsilon)t) + T(\epsilon)\sup_{s \in [0,T]}f(x(s)) + \int_T^t f(x(s))\,ds, \quad t \geq T(\epsilon).
	\]
	Let $x^* = 1 + \sup_{s \in [0,T]}x(s)$ and $x_+$ be the solution of 
	\begin{equation}\label{eq.x_+'_prime}
	x_+'(t) = K(1+\epsilon)(f\circ F^{-1})(K(1+\epsilon)t) + f(x_+(t)), \quad t \geq T(\epsilon); \quad 
	x_+(T)=x^*.
	\end{equation}
	By construction, $x(t)<x_+(t) \mbox{ for }t \geq T(\epsilon)$. As $\liminf_{t\to\infty}F(x(t))/t \geq 1$, there is $T_1(\epsilon)>T(\epsilon)$ such that $x_+(t) > x(t) > F^{-1}((1+2\epsilon)t) \mbox{ for }t \geq T_1(\epsilon)$. Thus
	$
	1/f(x_+(t)) < 1/(f \circ F^{-1})((1+2\epsilon)t) \mbox{ for } t \geq T_1(\epsilon),
	$
	and  
	\begin{equation}\label{eq.est_upper_x_+}
	\frac{x_+'(t)}{f(x_+(t))} < 1 + K(1+\epsilon)\frac{(f \circ F^{-1})(K(1+\epsilon)t)}{(f \circ F^{-1})((1+2\epsilon)t)}, \quad t \geq T_1(\epsilon).
	\end{equation}
	\begin{remark}\label{remark.superexp}
	Since limits of the following type arise frequently, we pause to remark that 
		\begin{equation}\label{eq.superexp}
		\lim_{t\to\infty}\frac{(f \circ F^{-1})((1-\epsilon)t)}{(f \circ F^{-1})(t)} = 0, \quad \mbox{for each }\epsilon\in (0,1),
		\end{equation}
		under \eqref{eq.f}, and we now give a proof of \eqref{eq.superexp}. By L'H\^{o}pital's rule,
		$
		\lim_{t\to\infty}\tfrac{d}{dt}F^{-1}(t)/F^{-1}(t) = \lim_{t\to\infty}(f \circ F^{-1})(t)/F^{-1}(t) = \infty. 
		$ 
		Integrating this asymptotic relation from $t-\eta$ to $t$ for $t$ sufficiently large gives 
		\begin{equation}\label{eq.lagged_limit_zero}
		\lim_{t\to\infty}\frac{F^{-1}(t - \eta)}{F^{-1}(t)} = 0, \quad \mbox{for each }\eta>0.
		\end{equation}
		For each \emph{fixed} $\epsilon\in (0,1)$ and $t$ sufficiently large, $t - 1 > (1-\epsilon)t$. Hence, letting $\eta = 1$ in \eqref{eq.lagged_limit_zero},
		$
		0 = \lim_{t\to\infty}F^{-1}(t-1)/F^{-1}(t) \geq \lim_{t\to\infty}F^{-1}((1-\epsilon)t)/F^{-1}(t), \mbox{ for each }\epsilon \in (0,1).
		$
		Now note that $x \mapsto f(x)/x$ being ultimately increasing implies
		$
		(f \circ F^{-1})((1-\epsilon)t)/F^{-1}((1-\epsilon)t) < (f \circ F^{-1})(t)/F^{-1}(t),  
		$ for $t$ large enough. But this is equivalent to 
		\[
		\frac{(f \circ F^{-1})((1-\epsilon)t)}{(f \circ F^{-1})(t)} < \frac{F^{-1}((1-\epsilon)t)}{F^{-1}(t)}, 
		\]
		and letting $t \to \infty$ yields the desired conclusion.
	\end{remark}
	From Remark \ref{remark.superexp} and \eqref{eq.est_upper_x_+}, we have
	$
	\limsup_{t\to\infty}x_+'(t)/f(x_+(t)) \leq 1.
	$
	Asymptotic integration now yields $\limsup_{t\to\infty}F(x_+(t))/t \leq 1$ and therefore $\limsup_{t\to\infty}F(x(t))/t \leq 1$, as required. The case $K = 0$ can be dealt with as above by replacing $F^{-1}(K(1+\epsilon)t)$ by $F^{-1}(\epsilon t)$ as appropriate.
	
	Conversely, $H(t)<x(t)$ for each $t \geq 0$. Hence
	$
	\limsup_{t\to\infty}F(H(t))/t \leq \lim_{t\to\infty}F(x(t))/t = 1.
	$
\end{proof}
\begin{theorem}\label{thm.limsup_big}
	Suppose \eqref{eq.f} and \eqref{eq.H} hold. Then the following are equivalent:
	\[
	(i.)\quad \limsup_{t\to\infty}\frac{F(H(t))}{t} = K \in (1,\infty), \qquad (ii.)\quad \limsup_{t\to\infty}\frac{F(x(t))}{t} = K \in (1,\infty).
	\]
\end{theorem}
\begin{proof}[Proof of Theorem \ref{thm.limsup_big}]
	First suppose $(i.)$ holds. Of course, $x(t) > H(t)$ for each $t \geq 0$, so we immediately have $\limsup_{t\to\infty}F(x(t))/t \geq \limsup_{t\to\infty}F(H(t)) = K$. By hypothesis, there exists $T(\epsilon)>0$ such that $H(t) < F^{-1}(K(1+\epsilon)t)$ for each $t \geq T(\epsilon)$. Follow the proof of Theorem \ref{thm.lim_preserved} to the definition of $x_+$ in \eqref{eq.x_+'_prime}. 
	Recalling Remark \ref{remark.superexp}, we have
	$
	\lim_{t\to\infty}(f \circ F^{-1})(K(1+\epsilon)t) / (f \circ F^{-1})(K(1+2\epsilon)t) = 0.
	$
	Thus there is $T_1(\epsilon)>T(\epsilon)$ such that
	\begin{equation}\label{eq.key_zero_limit}
	\frac{(f \circ F^{-1})(K(1+\epsilon)t)}{(f \circ F^{-1})(K(1+2\epsilon)t)}< \frac{2 \epsilon}{K(1+\epsilon)}, \quad t \geq T_1(\epsilon),
	\end{equation}
	for each $\epsilon\in (0,1)$. 	Let $x_u(t) = F^{-1}\left( K(1+2\epsilon)(t-T_1) + F^* \right)$ for $t \geq T_1(\epsilon)$ with $\bar{x}= x_+(T_1)$ and $F^*=1+\max\left(F(\bar{x}), KT_1(1+2\epsilon)\right)$. Our choices ensure $x_u(T_1)>x_+(T_1)>x(T_1)$ and 
	\begin{equation}\label{eq.f_circ_F_inv_est}
	(f \circ F^{-1})\left(  K(1+2\epsilon)(t-T_1) + F^* \right) > (f \circ F^{-1})\left( K(1+2\epsilon)t \right), \quad t \geq T_1(\epsilon).
	\end{equation}
	Now, using $K>1$, \eqref{eq.key_zero_limit}, and \eqref{eq.f_circ_F_inv_est}, we deduce that $x_u'(t)=K(1+2\epsilon)f(x_u(t))$ and moreover that 
	\begin{multline*}
	x_u'(t) - f(x_u(t)) -K(1+\epsilon) (f \circ F^{-1})(K(1+\epsilon)t) > \\ 2\epsilon\,(f \circ F^{-1})(K(1+2\epsilon)t) - K(1+\epsilon)(f \circ F^{-1})(K(1+\epsilon)t) > 0,
	\end{multline*}
	for each $t \geq T_1(\epsilon)$. Therefore
	$
	x_u'(t) > f(x_u(t)) + K(1+\epsilon) (f \circ F^{-1})(K(1+\epsilon)t) \mbox{ for } t \geq T_1(\epsilon),
	$
	and it follows from \eqref{eq.x_+'_prime} that $x(t) < x_+(t) < x_u(t)$ for each $t \geq T_1(\epsilon)$. Hence
	$	F(x(t)) < F(x_u(t)) = K(1+2\epsilon)(t-T_1)+F^*$.
	Dividing across by $t$, letting $t \to \infty$ and then $\epsilon \to 0^+$ in this inequality yields 
	$\limsup_{t\to\infty}F(x(t))/t \leq K$, as desired.
	
	Conversely, suppose $\limsup_{t\to\infty}F(x(t))/t = K > 1$. Since $x(t)>H(t)$ for each $t \geq 0$, it immediately follows that $\limsup_{t\to\infty}F(H(t))/t \leq K$. If $\limsup_{t\to\infty}F(H(t))/t = K^* \in (1,K)$, then the argument above can be repeated to show that $\limsup_{t\to\infty}F(x(t))/t \leq K^*$, a contradiction. If $K^* \in [0,1]$, the argument of Theorem \ref{thm.lim_preserved} similarly produces a contradiction. Therefore  $\limsup_{t\to\infty}F(H(t))/t = K$, as claimed.
\end{proof}
\begin{corollary}\label{thm.lim_big}
	Suppose \eqref{eq.f} and \eqref{eq.H} hold. Then
	\begin{equation}\label{thm.lim_big_conclusion}
	\lim_{t\to\infty}\frac{F(H(t))}{t} = K \in (1,\infty) \quad\mbox{implies}\quad\lim_{t\to\infty}\frac{F(x(t))}{t} = K.
	\end{equation}
\end{corollary}
A trivial lower bound shows that \eqref{thm.lim_big_conclusion} holds when $K=+\infty$, but this does not provide precise information on the rate of growth of $x$. The next result demonstrates that a condition implying 
$\lim_{t\to\infty} F(H(t))/t=+\infty$, and which yields a sharp characterisation of the growth rate is,
\begin{equation}\label{eq.int_f(H)_over_H}
\lim_{t\to\infty}\frac{\int_0^t f(H(s))\,ds}{H(t)} = 0.
\end{equation}
The condition $H'(t)/f(H(t))\to\infty$ as $t\to\infty$ yields \eqref{eq.int_f(H)_over_H}, can be easier to check, and can also be thought of as the limit (as $K\to\infty$) of the condition $H'(t)/f(H(t))\to K\in [0,\infty)$ as $t\to\infty$. This last condition yields $F(H(t))/t\to K$ as $t\to\infty$, which is the type of condition needed in Theorem~\ref{thm.limsup_big} and Corollary~\ref{thm.lim_big}.
 
\begin{theorem}\label{x_over_H_to_1}
Suppose \eqref{eq.f} and \eqref{eq.H} hold, and that $H$ is asymptotic to an increasing function $\tilde{H}$. Then \[
\lim_{t\to\infty}\frac{\int_0^t f(K \tilde{H}(s))\,ds}{\tilde{H}(t)} = 0\,\mbox{ for some }K> 1 \quad\mbox{ implies }\quad \lim_{t\to\infty}\frac{x(t)}{H(t)}=1.
\]
Conversely, $\lim_{t\to \infty}x(t)/H(t) = 1$ implies \eqref{eq.int_f(H)_over_H}.
\end{theorem}
\begin{proof}[Proof of Theorem \ref{x_over_H_to_1}]
First show that $\liminf_{t\to\infty}x(t)/H(t) < \infty$; suppose instead that $\lim_{t\to\infty}H(t)/x(t) = 0$. Hence
$
\lim_{t\to\infty}x(t)/\int_0^t f(x(s))\,ds = 1.
$
Thus there exists $T(\epsilon)>0$ such that 
	$
	(1-\epsilon)\int_0^t f(x(s))\,ds < x(t) < (1+\epsilon)\int_0^t f(x(s))\,ds \mbox{ for } t \geq T(\epsilon).
	$
	Define $J(t)= \int_0^t f(x(s))\,ds$ for $t \geq 0$, so that $J'(t) = f(x(t))$ for $t>0$. Hence $J'(t) = f(x(t)) < f((1+\epsilon)J(t))$ for $t > T(\epsilon)$. Thus, for $t > T(\epsilon)$,
	\[
	F\left( (1+\epsilon)J(t) \right) - F\left( (1+\epsilon)J(T) \right) =\int_{(1-\epsilon)J(T(\epsilon))}^{(1+\epsilon)J(t)}\frac{du}{f(u)} = \int_{T(\epsilon)}^t \frac{(1+\epsilon)J'(s)}{f((1+\epsilon)J(s))}\,ds \leq (1+\epsilon)(t-T(\epsilon)).
	\]
	Since $F$ is increasing,
	$
	F(x(t)) < F\left( (1+\epsilon)J(T(\epsilon)) \right) + (1+\epsilon)(t-T(\epsilon)), \mbox{ for } t > T(\epsilon),
	$
	and hence\\ $\limsup_{t\to\infty}F(x(t))/t \leq 1$. Analogously, $\liminf_{t\to\infty}F(x(t))/t \geq 1$. Therefore $\lim_{t\to\infty}F(x(t))/t = 1$. From the remark preceding Theorem \ref{x_over_H_to_1}, $\int_0^t f( K \tilde{H}(s))\,ds/\tilde{H}(t) \to 0$ as $t\to\infty$ implies $\lim_{t\to\infty}F(K \tilde{H}(t))/t = \infty$. Let $\eta \in (0,1)$ and suppose $f_1$, as defined in \eqref{eq.f}, is increasing for each $x \geq X(\eta)>1/\eta$. Then
	\[
F(t)\geq	F(\eta t) = \int_1^{\eta t}\frac{du}{f(u)} =  \eta \int_{1/\eta}^t \frac{ds}{f(\eta s)} = \int_{1/\eta}^t \frac{\eta s}{f(\eta s)}\frac{ds}{s} \geq \int_{X(\eta)}^t \frac{ds}{f(s)} = F(t) - F(X(\eta)), \quad t \geq X(\eta). 
	\]
	Hence $\lim_{t\to\infty}F(\eta t)/F(t)= 1$, since $\lim_{x\to\infty}F(x) = \infty$. Analogous arguments work for $\eta>1$ and therefore $\lim_{t\to\infty}F(\eta t)/F(t) = 1$ for each $\eta>0$; this limit and $\lim_{t\to\infty}F(K \tilde{H}(t))/t = \infty$ imply that $\lim_{t\to\infty}F(H(t))/t = \infty$. As $x(t)>H(t)$ for $t \geq 0$, $\lim_{t\to\infty}F(x(t))/t = \infty$, a contradiction. Hence, $\liminf_{t\to\infty}x(t)/H(t) < \infty$.
	
	Next we show that $\liminf_{t\to\infty} x(t)/H(t)=1$. Suppose not: let $\lambda=\liminf_{t\to\infty} x(t)/H(t)$. Since $x(t)>H(t)$, we have $\lambda\in (1,\infty)$ by supposition. Defining $J$ as above, we get $\liminf_{t\to\infty} J(t)/x(t)=1-\frac{1}{\lambda}>0$. Hence for every $\epsilon\in (0,1)$ there is a $T(\epsilon)>0$ such that $J(t)>(1-1/\lambda)(1-\epsilon)x(t)$ for $t\geq T(\epsilon)$. 
	Define $\Lambda_\epsilon$ by $\Lambda_\epsilon(\lambda-1)(1-\epsilon)=\lambda$.
	Since $J'(t)=f(x(t))$, $J'(t)<f\left(\Lambda_\epsilon J(t)\right)$ for $t>T(\epsilon)$. 	Define $J_\epsilon(t)=\Lambda_\epsilon J(t)>x(t)$ for $t\geq T(\epsilon)$. Then 
	$J_\epsilon'(t)/f(J_\epsilon(t))<\Lambda_\epsilon$ for $t>T(\epsilon)$. Asymptotic integration now yields	$\limsup_{t\to\infty} F(J_\epsilon(t))/t \leq \Lambda_\epsilon$.
	Using the fact that $x(t)<J_\epsilon(t)$ in the last limit, and then letting $\epsilon\to 0^+$, we have $\limsup_{t\to\infty} F(x(t))/t\leq \lambda/(\lambda-1)$.

	Next, let $\epsilon\in (0,1)$ be so small that $\lambda-\epsilon>1$. Then, by supposition, there 
	is $T'(\epsilon)>0$ such that $x(t)>(\lambda-\epsilon)H(t)>H(t)$ for $t\geq T'(\epsilon)$. Hence 
	$\limsup_{t\to\infty} F(H(t))/t \leq \lambda/(\lambda-1)$. But by hypothesis, $F(H(t))/t\to \infty$ as $t\to\infty$, a contradiction. Therefore, 
	$\liminf_{t\to\infty} x(t)/H(t)=1$, as claimed.  
	
	Finally, we show that $\limsup_{t\to\infty} x(t)/H(t)\leq 1$. By hypothesis, there is a $T_1(\epsilon)>0$ such that 
	\[
	(1-\epsilon)\tilde{H}(t)<H(t)<(1+\epsilon)\tilde{H}(t), 
	 \quad\mbox{for} t\geq T_1(\epsilon).\] Furthermore, because $\liminf_{t\to\infty} x(t)/\tilde{H}(t)=1$, there exists a sequence $(t_n)_{n \geq 1}$ such that $x(t_n)/\tilde{H}(t_n) < 1+\epsilon$ for $n \geq 1$. By supposition, for every $\epsilon\in (0,1)$, there is a $T_2(\epsilon)>0$ such that
	$
	\int_0^t f(K\tilde{H}(s))\,ds < \epsilon\tilde{H}(t)/2, \mbox{ for } t \geq T_2(\epsilon),
	$
	Now, let $\epsilon>0$ be so small that $\epsilon<\min(1,(K-1)/4)$. Set $B(\epsilon)=1+4\epsilon$ and $T(\epsilon) = \min\{t_n \,:\,t_n > T_1(\epsilon)+ T_2(\epsilon) \}$. Define $x_+(t)=B(\epsilon)\tilde{H}(t)$ for $t\geq T(\epsilon)$ and notice that 
	$
	x_+(T)=(1+4\epsilon)\tilde{H}(T)>(1+\epsilon)\tilde{H}(T)>x(T).
	$
	Next,
	$\int_{T(\epsilon)}^t f(x_+(s))\,ds \leq \int_0^t f(K\tilde{H}(s))\,ds <\epsilon\tilde{H}(t)/2$ for $ t\geq T(\epsilon)$.	Using this estimate, the monotonicity of $\tilde{H}$, and $H(T)>(1-\epsilon)\tilde{H}(T)$, we obtain
	\begin{align}
	x(T) - H(T) + (1+\epsilon)\tilde{H}(t)+\int_T^t f(x_+(s))\, ds 
	\label{eq.x+c} 
	&\leq (1+7\epsilon/2) \tilde{H}(t) < x_+(t), \quad t\geq T(\epsilon).
	\end{align}
	On the other hand, since $H(t)<(1+\epsilon)\tilde{H}(t)$ for $t\geq T(\epsilon)$, it follows that
	\begin{equation} \label{eq.xc}
	x(t)=x(T)+H(t)-H(T)+\int_T^t f(x(s))\,ds < x(T) - H(T) + (1+\epsilon)\tilde{H}(t)+\int_T^t f(x(s))\, ds.
	\end{equation}
	Since \eqref{eq.x+c}, \eqref{eq.xc}, and $x(T)<x_+(T)$ hold, a comparison argument using the monotonicity of $f$ gives 
	$x(t)/H(t) < x_+(t)/H(t)=1+4\epsilon$ for $t \geq T(\epsilon)$. Therefore $\limsup_{t\to\infty}x(t)/H(t) \leq 1$, whence the claimed limit.  
	
	For the converse, note that $\lim_{t\to\infty}x(t)/H(t) = 1$, implies $\lim_{t \to \infty}\int_0^t f(x(s))\,ds / H(t) = 0$. Since $x(t) \geq H(t)$ for $t \geq 0$, $\int_0^t f(x(s))\,ds \geq \int_0^t f(H(s))\,ds$ for $t \geq 0$. This estimate and the last limit prove the claim. 
\end{proof}
\begin{definition}\label{defn_O-reg}
	A nonnegative measurable function $\phi$ is called $O$--regularly varying if
	\[
	0 < \liminf_{x\to\infty} \phi(\lambda x)/\phi(x) \leq \limsup_{x\to\infty} \phi(\lambda x)/\phi(x) < \infty, \mbox{ for each }\lambda>1.
	\]
\end{definition}
\noindent While Definition \ref{defn_O-reg} appears restrictive, if $\limsup_{x\to\infty}\phi(\lambda x)/\phi(x)$ is finite for some $\lambda >1$ and  $\phi$ is increasing, then $\phi$ is $O$--regularly varying \cite[Corollary 2.0.6, p.65]{bingham1989regular}. We can now state a simple corollary to Theorem \ref{x_over_H_to_1}.
\begin{corollary}
	Suppose \eqref{eq.f} and \eqref{eq.H} hold, and that $H$ is asymptotic to an increasing function $\tilde{H}$. If $f$ is $O$--regularly varying, then the following are equivalent:
	\[
	(i.)\quad \lim_{t\to\infty}\frac{\int_0^t f(H(s))\,ds}{H(t)} = 0, \qquad (ii.)\quad \lim_{t\to\infty}\frac{x(t)}{H(t)}=1.
	\]
\end{corollary}
\begin{examples}
Choose $f(x) = (x+e)\log(x+e)$ for $x \geq 0$. Straightforward estimation shows that $F(x) \sim \log\log(x)$ as $x \to \infty$. Let $H(t) = \exp\exp(Kt^\alpha)- e$ for $t \geq 0$, with $\alpha>0$ and $K>1$. If $\alpha \in (0,1)$, $\limsup_{t\to\infty}F(H(t))/t = 0$ and Theorem \ref{thm.lim_preserved} implies that $\lim_{t\to\infty}\log\log(x(t))/t = 1$. If $\alpha = 1$, then $\lim_{t\to\infty}F(H(t))/t = K>1$ and by Corollary \ref{thm.lim_big}, $\lim_{t\to\infty}\log\log(x(t))/t = K$. Finally, when $\alpha>1$, $\limsup_{t\to\infty}F(H(t))/t = \infty$ and Theorems \ref{thm.lim_preserved} and \ref{thm.limsup_big} do not apply. However, $\lim_{t\to\infty}\int_0^t f(H(s))\,ds/H(t) = 0$ if $\alpha>1$, so Theorem \ref{x_over_H_to_1} implies $x(t) \sim H(t)$ as $t\to\infty$.
\end{examples}
\section{Fluctuation results}
Finally, we sketch a result which applies when $H$ fluctuates rather than grows, but the size of the large fluctuations is known. We assume that the fluctuations are large by imposing the conditions of Theorem \ref{x_over_H_to_1} on a growing function $\gamma$ which tracks the largest fluctuation size, and impose symmetry in the following manner:
\begin{equation} \label{eq.symm}
\limsup_{t\to\infty} \frac{H(t)}{\gamma(t)}=1, \quad \liminf_{t\to\infty} \frac{H(t)}{\gamma(t)}=-1, 
\quad \lim_{|x|\to \infty } \frac{|f(x)|}{\varphi(|x|)}=1.
\end{equation}
$\varphi$ satisfies \eqref{eq.f} and obeys $\int_1^\infty du/\varphi(u)=+\infty$, so it plays the role of $f$ in earlier results. 
We can prove analogues of Theorems~\ref{thm.lim_preserved} and \ref{thm.limsup_big}, with ``small'' $\gamma$, but here $\gamma$ is ``large'' relative to the nonlinearity; more precisely:
\begin{equation} \label{eq.gamma}
\text{There exists $K>1$ such that } \lim_{t\to\infty} \frac{\int_0^t \varphi(K\gamma(s))\,ds}{\gamma(t)}=0, \quad \text{$\gamma\in C^1((0,\infty),(0,\infty))$ is increasing}. 
\end{equation}
The technical condition $\gamma\in C^1(0,\infty)$ simplifies the proof of the next result, which is an analogue of Theorem~\ref{x_over_H_to_1}.
\begin{theorem}\label{thm.gamma}
	If \eqref{eq.symm} and \eqref{eq.gamma} hold with $\varphi$ satisfying \eqref{eq.f} and $f \in C(\mathbb{R};\mathbb{R})$, then
	\[
	\lim_{t\to\infty} \frac{x(t)-H(t)}{\gamma(t)}=0, \quad \limsup_{t\to\infty} \frac{x(t)}{\gamma(t)}=1, \quad \liminf_{t\to\infty} \frac{x(t)}{\gamma(t)}=-1.
	\]
\end{theorem}
\begin{proof}
	The second and third limits are an easy consequence of the first limit, and the first two limits in \eqref{eq.symm}. It remains therefore to prove the first limit. For every $\epsilon>0$ there is $A(\epsilon)>0$ such that $|f(x)|\leq A(\epsilon)+(1+\epsilon)\varphi(|x|)$ for all $x\in\mathbb{R}$. For every $\epsilon>0$, there is $T_1(\epsilon)$ such that 
	$|H(t)|\leq (1+\epsilon)\gamma(t)$ for $t\geq T_1(\epsilon)$. Since $\gamma(t)\to\infty$ and $\varphi$ is increasing, estimating the integral in \eqref{eq.gamma} gives $t/\gamma(t)\to 0$ as $t\to\infty$. Hence there is $T_2(\epsilon)>0$ such that $A(\epsilon)t<\epsilon \gamma(t)$ for $t\geq T_2(\epsilon)$. Let $T=\max(T_1,T_2)$, $C(\epsilon)=|x(T)|+|H(T)|$, and integrate \eqref{eq.ODE} to obtain 
	\[
	x(t)=x(T)+H(t)-H(T)+\int_T^t f(x(s))\,ds,\quad t\geq T(\epsilon).
	\]  
	Using the estimates above, we arrive at the inequality
	$
	|x(t)|\leq C(\epsilon)+(1+2\epsilon)\gamma(t)+(1+\epsilon)\int_T^t \varphi(|x(s)|)\,ds, 
	$ for $t\geq T.$ Since $\gamma \in C^1$, we may define $x_+$ to be the solution of
	\[
	x_+'(t)= (1+2\epsilon)\gamma'(t)+(1+\epsilon) \varphi(x_+(t)), \quad t\geq T;\quad x_+(T)=1+C(\epsilon)+(1+2\epsilon)\gamma(T).
	\]
	Then $|x(t)|< x_+(t)$ for $t\geq T$. 
	Let $\tilde{x}(t)=x_+(t+T)$ for $t\geq 0$ and define $\tilde{h}(t)= (1+2\epsilon)\gamma'(t+T)$ for $t\geq 0$. Hence
	\[\tilde{x}'(t)= \tilde{h}(t)+(1+\epsilon)\varphi(\tilde{x}(t)), \quad t \geq 0; \quad \tilde{x}(0)=1+C(\epsilon)+(1+2\epsilon)\gamma(T).\]
	Then, applying Theorem~\ref{x_over_H_to_1} to $\tilde{x}$ we get $\tilde{x}(t)/\int_0^t \tilde{h}(s)\,ds\to 1$ 
	as $t\to\infty$. This leads quickly to $x_+(t)/\gamma(t)\to 1+2\epsilon$ as $t\to\infty$. Therefore
	$\limsup_{t\to\infty} |x(t)|/\gamma(t)\leq 1$. Thus, for every $\epsilon<K-1$, there is $T_3(\epsilon)>0$ such that $|x(t)|\leq (1+\epsilon)\gamma(t)$ 
	for $t\geq T_3(\epsilon)$. Define $I(t)=\int_0^t f(x(s))\,ds$ for $t\geq 0$. Then, for $t\geq T_3$, 
	\begin{align*}
	|I(t)|&\leq |I(T_3)|+\int_{T_3}^t \{A(\epsilon)+(1+\epsilon)\varphi(|x(s)|)\}\,ds
	\leq |I(T_3)|+A(\epsilon) t +(1+\epsilon)\int_{0}^t\varphi(K\gamma(s))\,ds.
	\end{align*}
	Now dividing by $\gamma(t)$, the last term on the righthand side tends to $0$ as $t\to\infty$ by \eqref{eq.gamma}, as does the second term since $t/\gamma(t)\to 0$ as $t\to\infty$. Therefore $\lim_{t\to\infty}I(t)/\gamma(t)= 0$. Since $x(t)-H(t)=x(0)+I(t)$ for $t\geq 0$, the first limit in the claim follows, which completes the proof.  
\end{proof}
Using Theorem \ref{thm.gamma} with $H(t)=\int_0^t \sigma(s)\,dB(s)$ in \eqref{eq.SDE} and $\gamma = \Sigma$, where
\begin{equation}\label{eq.Sigma}
\Sigma(t)= \sqrt{ 2 \left(\int_0^t \sigma^2 (s)ds\right) \log\log\left( \int_0^t \sigma^2 (s)ds \right)}, \quad t \geq 0,
\end{equation}
the law of the iterated logarithm for continuous martingales can be used to show the following fluctuation result regarding solutions to \eqref{eq.SDE} (see \cite[Ch. V, Ex. 1.15]{revuzandyor}).
\begin{corollary}\label{corollary.SDE}
Let $X$ be the unique strong solution to \eqref{eq.SDE}, $f \in C(\mathbb{R};\mathbb{R})$, and $\sigma \notin L^2([0,\infty);\mathbb{R})$. Suppose that \eqref{eq.symm} and \eqref{eq.gamma} hold with $\gamma=\Sigma$, as defined by \eqref{eq.Sigma}, and $\varphi$ satisfies \eqref{eq.f}. Then, with probability one, $X$ obeys
\[
\liminf_{t\to\infty}\frac{X(t)}{\Sigma(t)} = -1, \quad \limsup_{t\to\infty}\frac{X(t)}{\Sigma(t)}=1.
\]
\end{corollary}
\bibliography{superlinear_ode_refs}

\begin{thebibliography}{10}

\bibitem{agarwal2005generalization}
R.~P. Agarwal, S.~Deng, and W.~Zhang.
\newblock Generalization of a retarded {G}ronwall--like inequality and its
  applications.
\newblock {\em Applied Mathematics and Computation}, 165(3):599--612, 2005.

\bibitem{appleby2014necessary}
J.~A.~D. Appleby and D.~D. Patterson.
\newblock On necessary and sufficient conditions for preserving convergence
  rates to equilibrium in deterministically and stochastically perturbed
  differential equations with regularly varying nonlinearity.
\newblock In {\em Recent Advances in Delay Differential and Difference
  Equations}, pages 1--85. Springer, 2014.

\bibitem{app_patt_ejde}
J.~A.~D. Appleby and D.~D. Patterson.
\newblock Hartman--{W}intner growth results for sublinear functional
  differential equations.
\newblock {\em Electron. J. Differential Equations}, 2017(21):1--45, 2017.

\bibitem{bihari1956generalization}
I.~Bihari.
\newblock A generalization of a lemma of {B}ellman and its application to
  uniqueness problems of differential equations.
\newblock {\em Acta Mathematica Hungarica}, 7(1):81--94, 1956.

\bibitem{bingham1989regular}
N.~H. Bingham, C.~M. Goldie, and J.~L. Teugels.
\newblock {\em Regular variation}, volume~27.
\newblock Cambridge University Press, 1989.

\bibitem{brunner2012blow}
H.~Brunner and Z.~Yang.
\newblock Blow-up behavior of {H}ammerstein--type {V}olterra integral
  equations.
\newblock {\em J. Integral Equations Appl}, 24(4):487, 2012.

\bibitem{MR1081393}
A.~Constantin.
\newblock A {G}ronwall-like inequality and its applications.
\newblock {\em Atti Accad. Naz. Lincei Cl. Sci. Fis. Mat. Natur. Rend. Lincei
  (9) Mat. Appl.}, 1(2):111--115, 1990.

\bibitem{MR1177923}
A.~Constantin.
\newblock On pointwise estimates for solutions of {V}olterra integral
  equations.
\newblock {\em Boll. Un. Mat. Ital. A (7)}, 6(2):215--225, 1992.

\bibitem{H}
P.~Hartman.
\newblock {\em Ordinary differential equations}.
\newblock SIAM, Philadelphia, second edition, 2002.

\bibitem{malolepszy2008blow}
T.~Ma{\l}olepszy and W.~Okrasi{\'n}ski.
\newblock Blow-up conditions for nonlinear {V}olterra integral equations with
  power nonlinearity.
\newblock {\em Applied Mathematics Letters}, 21(3):307--312, 2008.

\bibitem{MR0372135}
B.~G. Pachpatte.
\newblock On some generalizations of {B}ellman's lemma.
\newblock {\em J. Math. Anal. Appl.}, 51:141--150, 1975.

\bibitem{MR1487077}
B.~G. Pachpatte.
\newblock {\em Inequalities for Differential and Integral Equations}, volume
  197 of {\em Mathematics in Science and Engineering}.
\newblock Academic Press, Inc., San Diego, CA, 1998.

\bibitem{pinto1990integral}
M.~Pinto.
\newblock Integral inequalities of bihari-type and applications.
\newblock {\em Funkcialaj Ekvacioj}, 33(3):387--403, 1990.

\bibitem{revuzandyor}
D.~Revuz and M.~Yor.
\newblock {\em Continuous martingales and Brownian motion}, volume 293.
\newblock Springer Science \& Business Media, 1999.

\bibitem{roberts2007recent}
C.~A. Roberts.
\newblock Recent results on blow-up and quenching for nonlinear {V}olterra
  equations.
\newblock {\em Journal of Computational and Applied Mathematics},
  205(2):736--743, 2007.

\end{thebibliography}
\bibliographystyle{abbrv}
\end{document}